\documentclass[11pt]{article}

\usepackage{amsmath,amsfonts,amssymb,amsthm,graphics,amscd,enumerate,verbatim,color,comment}

\newtheorem{theorem}{Theorem}[section]
\newtheorem{definition}[theorem]{Definition}
\newtheorem{proposition}[theorem]{Proposition}

\newtheorem{corollary}[theorem]{Corollary}
\newtheorem{remark}[theorem]{Remark}

\def\*<{\,*\!\!<}
\def\<*{<\!\!*\,}
\def\R{{\rm{Im}\,}}
\def\N{{\rm{Ker}\,}}

\def\G{{\mathcal{G}}}

\begin{document}

\vspace{2.5cm}

\small{\noindent This manuscript was submitted in  "Revista de la Real Academia de Ciencias Exactas,  F\'isicas y Naturales. Serie A. Matem\'aticas" on March 21, 2022.}

\vspace{1.5cm}

\begin{center}
\Large{\bf 1MP and MP1 inverses and one-sided star orders \\
in a ring with involution}
\end{center}

\begin{center}
{\bf Dragan S. Raki\'c and Martin Z. Ljubenovi\'c} \footnote{The research is financially supported by the  Ministry of Education, Science and Technological Development, Republic of Serbia, Grant No. 451-03-68/2022-14/ 200109 and by the bilateral project between Serbia and Slovenia (Generalized inverses, operator equations and applications, Grant No. 337-00-21/2020-09/32).}
\end{center}

\begin{abstract}
The classes of 1MP-inverses and MP1-inverses are recently introduced classes of generalized inverses of complex matrix. Actually, they coincide with the classes of $\{1,2,3\}$ and $\{1,2,4\}$ inverses, respectively.
We consider these inverses in the context of a ring with involution and prove that their most important characterizations and properties remain true.
We show that the binary relations based on these inverses are in fact the well known left-star and right-star partial orders.
We extend these relations to the ring case, connect them with the unified theory of partial order relations based on generalized inverses and provide several properties. Finally, we indicate how these results can be applied to bounded Hilbert space operators.

\medskip

{\it  Mathematics Subject Classification\/}: 16U90, 06A06, 15A09.
\smallskip

{\it Keywords and phrases\/}: One-sided star partial order, 1MP-inverse, $\{1,2,3\}$-inverse, Ring with involution
\end{abstract}

\section{Introduction and preliminaries}

Let $A\in \mathbb{C}^{m\times n}$, where $\mathbb{C}^{m\times n}$ is the set of $m\times n$ complex matrices. We denote by $A^*$, $\R A$ and $\N A$ the conjugate transpose, column space and null space of $A$ respectively.
A matrix $A^-\in \mathbb{C}^{n\times m}$ is called a $g$-inverse of $A$ if the equation $AA^-A=A$ is satisfied.
The set of all $g$-inverses of $A$ is denoted by $A\{1\}$.
If $AXA=A$ and $XAX=X$ then $X\in \mathcal{C}^{n\times m}$ is called a reflexive $g$-inverse of $A$.
Recall that the Moore-Penrose inverse of $A$ is the unique matrix $A^\dag\in\mathbb{C}^{n\times m}$ which satisfies the equations
$$(1)\;\, AA^\dag A=A\quad (2)\;\, A^\dag AA^\dag=A^\dag \quad (3) \;\,(AA^\dag)^*=AA^\dag \quad (4)\;\, (A^\dag A)^*=A^\dag A.$$

Hern\'andez et al. in \cite{1MP} considered the following equivalence relation on $A\{1\}$.
For $A^-, A^= \in A\{1\}$,
$$A^-\sim_l A^= \;\; \Leftrightarrow\;\; A^-A=A^=A.$$
It turns out that the simplest representative (relative to singular value decomposition of $A$) of the class $[A^-]_{\sim_l}$ is given by the matrix $A^-AA^\dag$.
The authors in \cite{1MP} defined a new class of $g$-inverses of $A$ as the set of all canonical representatives of the quotient space $A\{1\}/\sim_l$.
They analogously defined the dual class as well.
\begin{definition}(\cite{1MP}).
Let $A\in\mathbb{C}^{m\times n}$. For each $A^-\in A\{1\}$, the matrix $A^{-\dag}=A^-AA^\dag$ is called a 1MP-inverse of $A$ and the set of all 1MP-inverses of $A$ is denoted by $A\{-\dag\}$. Similarly, the matrix $A^{\dag -}=A^\dag AA^-$ is called a MP1-inverse of $A$ and the set of all MP1-inverses of $A$ is denoted by $A\{\dag -\}$.
\end{definition}

Among other applications, $g$-inverses are used to define nice binary relations. This is the case with 1MP and MP1-inverses as well. For $A,B\in\mathbb{C}^{m\times n}$ the relation $<^{-\dag}$ was defined in \cite{1MP} in the following way:
\begin{equation}\label{eq15<-+}
A<^{-\dag}B \; \Leftrightarrow \; AA^{-\dag}=BA^{-\dag} \, \text{ and } \; A^{-\dag}A=A^{-\dag}B \, \text{ for some } \, A^{-\dag}\in A\{-\dag\}.
\end{equation}
Similarly,
\begin{equation}\label{eq15<+-}
A<^{\dag -}B \; \Leftrightarrow \; AA^{\dag -}=BA^{\dag -} \, \text{ and } \; A^{\dag -}A=A^{\dag -}B \, \text{ for some } \, A^{\dag -}\in A\{\dag -\}.
\end{equation}

Our aim is to generalize and investigate the 1MP and MP1 inverses and related binary relations (\ref{eq15<-+}) and (\ref{eq15<+-}) in the context of an arbitrary ring with involution.
We show that these notions are actually $\{1,2,3\}$-inverse, $\{1,2,4\}$-inverse, the left-star order and the right-star order. Although the ring setting is more general then complex matrix setting, we will prove almost all known important properties of these inverses and relations. It shows that the nature of these notions is, to a large degree, purely algebraic.

It is clear that in the proofs in the ring case, we cannot use the standard linear algebra techniques which are dominant in the matrix case. So we use different approaches and methods. One approach is the matrix representation of ring elements with respect to appropriate set of idempotents. This approach is described at the end of this section.

This paper is organized as follows.

By the end of this section we will give the basic definitions and set up notation.

In Section 2 we will focus on the complex matrix case. We will first note that the set $A\{-\dag\}$ coincides with the set of all reflexive least square $g$-inverses of $A$ and that $A\{\dag -\}$ is equal to the set of all reflexive minimum norm $g$-inverses of $A$. We will give several other characterizations of these inverses. The most important result in Section 2 is Theorem \ref{th152} where we show that the relations $<^{-\dag}$ and $<^{\dag -}$ are actually the well-known left-star partial order ($\*<$) and right star partial order ($\<*$), respectively. Recall that these relations was defined by Baksalary and Mitra in \cite{BaksalaryMitra}:
\begin{align}\label{eq15*<}
A\*< B&  \; \Leftrightarrow \; A^*A=A^*B  \, \text{ and } \; \R A\subseteq \R B \\
A\<* B&  \; \Leftrightarrow \; AA^*=BA^*  \, \text{ and } \; \R A^*\subseteq \R B^* \nonumber.
\end{align}

In Section 3 we show that 1MP-inverse can be introduced for elements of arbitrary ring with involution in the same way as it is done in \cite{1MP} for complex matrices. We show that
\begin{align*}
a\{-\dag\}&=a\{1,2,3\}\\
&=\{h+(1-ha)wah : w\in R\} \\
&=\{(a^*a)^-a^*:(a^*a)^-\in (a^*a)\{1\}\}
\end{align*}
like in the matrix case.
So, the most important characterizations stay valid in the ring case.

The definition of the relation $<^{-\dag}$ is algebraic, so we can extend it in the ring context. In Section 4, we introduce the left-star order relation $\*<$ in a ring by analogy with complex matrix case. Then we show that $a<^{-\dag}b$ if and only if $a\*< b$ when $a,b\in R^{(1,3)}$, (see below for notation). We will examine this order through the Mitra's unified theory of partial orders based on generalized inverses, \cite{Unified_theory} and \cite{Jedinstvena_teorija}. In particular, we present in Theorem \ref{th158} very useful simultaneous diagonalizations of $a$ and $b$ when $a\*< b$. This result has interesting interpretation when $a$ and $b$ are bounded Hilbert space operators with closed ranges.
Of course, we will prove that $\*<$ is indeed the partial order relation on $R^{(1,3)}$. Moreover, we will show in Theorem \ref{th157} that
$$a\*< b \; \Leftrightarrow \;b\{1,3\}\subseteq a\{1,3\}$$
which is highly nontrivial result even in the complex matrix case, \cite{BaksalaryMitra}.

In the Section 5 we present the dual results for MP1-inverse (ie. $\{1,2,4\}$-inverse) and for $<^{\dag -}$ relation (ie. right-star order).

In the last Section 6, we will give adequate interpretations of the presented results in the case of bounded Hilbert space operators.

We will now give some preliminaries.
Let $R$ denotes a ring with involution $*$ and the multiplicative identity 1. Like in the matrix case, if $a\in R$ and there exist $x\in R$ such that $axa=a$ then we say that $a$ is regular and $x$ is $g$-inverse of $a$. If in addition $xax=x$ then $x$ is called a reflexive $g$-inverse of $a$.
The Moore-Penrose inverse of $a$ is the element $a^\dag\in R$ which satisfies the equations
$$(1)\;\; axa=a\quad\quad (2)\;\; xax=x\quad\quad (3)\;\; (ax)^*=ax\quad\quad (4)\;\; (xa)^*=xa.$$
The Moore-Penrose inverse is unique in the case when it exists.
If $x$ satisfies equations $i_1, i_2, \dots, i_n$ then $x$ is called $\{i_1, i_2, \dots, i_n\}$-inverse of $a$. By $a\{i_1, i_2, \dots, i_n\}$ and $R^{(i_1, i_2, \dots, i_n)}$ we denote respectively the set of all $\{i_1, i_2, \dots, i_n\}$-inverses of $a$ and the set of all elements in $R$ that possess a $\{i_1, i_2, \dots, i_n\}$-inverse. For short, we will denote by $R^\dag$ the set of all Moore-Penrose invertible elements in $R$.

As usual, for $a\in R$, $Ra=\{xa:x\in R\}$ and $aR=\{ax:x\in R\}$.

An element $e\in R$ is idempotent (self-adjoint idempotent) if $e^2=e$ ($e^2=e=e^*$).
\begin{definition}\label{def15p,q_invertibilnost} 
  Let $p$ and $q$ be idempotents in $R$.
  We say that an element $a\in R$ is $(p,q)$-invertible if $a\in pRq$ and there exists $a^-_{p,q}\in qRp$ such that
$$aa^-_{p,q}=p \;\; \text{ and} \;\; a^-_{p,q}a=q.$$
 In this case we say that $a^-_{p,q}$ is the $(p,q)$-inverse of $a$.
\end{definition}
It is easy to see that for $a\in R$ there exist idempotents $p$ and $q$ such that $a$ is $(p,q)$-invertible if and only if $a$ is regular. When it exists, the $(p,q)$-inverse is unique.
If $a$ is $(p,q)$-invertible and $b\in pR$ then the equation $ax=b$ has the unique solution in the set $qR$, namely $x=a^-_{p,q}b$. Similarly, if $b\in Rq$ then the equation $xa=b$ has the unique solution in the set $Rp$ given by $x=ba^-_{p,q}$.

The idempotents $e_1,e_2,\dots,e_n\in R$ are called mutually orthogonal if they are orthogonal in pairs, that is, if $e_ie_j=0$ for $i\neq j$. If $e_1,e_2,\dots ,e_n\in R$ are mutually orthogonal idempotents such that
\begin{equation}\label{eq15sumaIdemp1}
1=e_1+e_2+\dots +e_n
\end{equation}
then the equality (\ref{eq15sumaIdemp1}) is called a decomposition of the identity of the ring $R$. The decomposition of the identity is orthogonal if $e_i$, $i=1,\dots,n$ are self-adjoint.
The following observation was provided in \cite{Decomposition_of_identity}.
Let $1=e_1+\dots+e_m$ and $1=f_1+\dots+f_n$ be two decompositions of the identity of a ring $R$. For any $x\in R$ we have
$$x=\left(\sum\limits_{i=1}^m e_i\right)x\left(\sum\limits_{j=1}^n f_j\right)=\sum\limits_{i=1}^m\sum\limits_{j=1}^n e_ixf_j= \bmatrix x_{11} & \cdots & x_{1n} \\ \vdots & \ddots & \vdots \\ x_{m1} & \cdots & x_{mn}\endbmatrix_{e\times f},$$
where $x_{ij}=e_ixf_j$. Note that the representation of an $x$ as $x=[x_{ij}]_{e\times f}$ is unique.

If $y=[y_{ij}]_{e\times f}$ then $x+y$ can be interpreted as addition of two matrices over $R$.
Let $1=g_1+\cdots +g_k$ be another decomposition of the identity of $R$ and let $z=[z_{jl}]_{f\times g}$, $z_{jl}=f_jzg_l$. As $f_if_j=0$ for $i\neq j$, it is easy to see that the product $xz$ can be calculated as the multiplication of two matrices over $R$. Also,
   \begin{equation}\label{eq15matr_rep_za_x_zvezda}
x^*=\left[
      \begin{array}{ccc}
        x_{11}^* & \cdots & x_{m1}^* \\
        \vdots & \ddots & \vdots \\
        x_{1n}^* & \cdots & x_{mn}^*
      \end{array}
      \right]_{f^*\times e^*},
\end{equation}
where the above representation is with respect to decompositions of the identity $1=f_1^*+\cdots +f_n^*$ and $1=e_1^*+\cdots +e_m^*$.

Let $e,f\in R$ be two idempotents. They induce two decompositions of the identity $1=e+(1-e)$ and $1=f+(1-f)$. Then we will write $x\in R$ in the following way
$$x=\bmatrix x_1 & x_2 \\ x_3 & x_4 \endbmatrix_{e\times f}.$$

\section{1MP-inverse and $<^{-\dag}$ partial order in $\mathbb{C}^{m\times n}$}

Recall that for $A\in \mathbb{C}^{m\times n}$ a matrix $A_l^-\in \mathbb{C}^{n\times m}$ is a least squares $g$-inverse of $A$ if $A_l^-b$ is a least squares solution of $Ax=b$ for all $b\in \mathbb{C}^m$, that is, if for all $b\in \mathbb{C}^m$ the $\ell_2$-norm of $Ax-b$ is smallest when $x=A_l^-b$. A matrix $A_m^-\in \mathbb{C}^{n\times m}$ is a minimum norm $g$-inverse of $A$ if $A^-_m b$ provides a solution with minimum norm of equation $Ax=b$ whenever it is consistent.

The following characterizations of 1MP and MP1 inverses are given in \cite{1MP}.

\begin{theorem}\label{th151} (\cite[Theorem 3.1]{1MP})
For $A\in \mathbb{C}^{m\times n}$ the following hold
\begin{align*}
A\{-\dag\}&=A\{1,2,3\}=\{X\in\mathbb{C}^{n\times m}: AXA=A, XAX=X, (AX)^*=AX\} \\
A\{\dag -\}&=A\{1,2,4\}=\{X\in\mathbb{C}^{n\times m}: AXA=A, XAX=X, (XA)^*=XA\}.
\end{align*}
\end{theorem}

\begin{theorem}\label{th151MPrls}
Let $A\in \mathbb{C}^{m\times n}$. The set of all 1MP-inverses of $A$ is equal to the set of all reflexive least square $g$-inverses of $A$. The set of all MP1-inverses of $A$ is equal to the set of all reflexive minimum norm $g$-inverses of $A$. Moreover,
\begin{align}\label{eq15ch_rls_rmn}
\begin{split}
A\{-\dag\}&=\{(A^*A)^-A^*:(A^*A)^-\in (A^*A)\{1\}\} \; \text{ and } \\
A\{\dag-\}&=\{A^*(AA^*)^-:(AA^*)^-\in (AA^*)\{1\}\}.
\end{split}
\end{align}
\end{theorem}
\begin{proof}
By \cite[Theorem 2.5.14]{knjiga}, $X$ is a least squares $g$-inverse of $A$ if and only if $AXA=A$ and $(AX)^*=AX$. Similarly, $X$ is a minimum norm $g$-inverse of $A$ if and only if $AXA=A$ and $(XA)^*=XA$, \cite[Theorem 2.5.5]{knjiga}.
Now the first part of the theorem follows by Theorem \ref{th151} while characterizations in (\ref{eq15ch_rls_rmn}) follow by \cite[Theorem 2.5.19 and Theorem 2.5.9]{knjiga}.
\end{proof}

We already know by Theorem \ref{th151} that $A\{-\dag\}=A\{1,2,3\}$. We give some more characterizations of 1MP and MP1-inverses.

\begin{theorem}\label{th1516}
For $A\in \mathbb{C}^{m\times n}$ the following conditions are equivalent:
\begin{enumerate}[(i)]
\item $X$ is a 1MP-inverse of $A$.
\item $X\in A\{1,2\}$ and $\N X=\N A^*$.
\item $X\in A\{1\}$ and $\N A^*\subseteq \N X$.
\item $X\in A\{2\}$ and $\N X\subseteq \N A^*$.
\end{enumerate}
\end{theorem}
\begin{proof}
$(i)\Rightarrow (ii)$: If $X$ is a 1MP-inverse of $A$ then $AXA=A$, $XAX=X$ and $(AX)^*=AX$. From $A^*=(AXA)^*=A^*AX$ and $X=XAX=X(AX)^*=XX^*A^*$, we conclude that $\N X=\N A^*$.

$(ii)\Rightarrow (iii)$ and $(ii)\Rightarrow (iv)$ are trivial.

$(iii)\Rightarrow (i)$: Suppose that $AXA=A$ and $\N A^*\subseteq \N X$. From \cite[Lemma 2.1]{prvi_rad} we know that $\N A^*\subseteq \N X$ implies $X=X(A^*)^\dag A^*=X(AA^\dag)^*=XAA^\dag$. Therefore, $AX=AA^\dag$ which is self-adjoint. Also, $XAX=XAA^\dag=X$, so $X\in A\{1,2,3\}$.

$(iv)\Rightarrow (i)$: Suppose that $XAX=X$ and $\N X\subseteq \N A^*$. Again, by \cite[Lemma 2.1]{prvi_rad}, we have $A^*=A^*X^\dag X$, that is, $A=X^\dag XA$. It follows that $AX=X^\dag X$ and $AXA=X^\dag X A=A$, so $X\in A\{1,2,3\}$.
\end{proof}

We can prove likewise the dual result for the MP1-inverse.

\begin{theorem}
For $A\in \mathbb{C}^{m\times n}$ the following conditions are equivalent:
\begin{enumerate}[(i)]
\item $X$ is a MP1-inverse of $A$.
\item $X\in A\{1,2\}$ and $\R X=\R A^*$.
\item $X\in A\{1\}$ and $\R X\subseteq \R A^*$.
\item $X\in A\{2\}$ and $\R A^*\subseteq \R X$.
\end{enumerate}
\end{theorem}

We now turn to order relations defined by 1MP and MP1 inverses. Next theorem shows that the relations $<^{-\dag}$ and $<^{\dag -}$ defined in (\ref{eq15<-+}) and (\ref{eq15<+-}) are respectively the left-star and right-star order relations defined in (\ref{eq15*<}).
\begin{theorem}\label{th152}
Let $A,B\in \mathbb{C}^{m\times n}$. Then $A<^{-\dag} B$ if and only if $A\*< B$. Also, $A<^{\dag -}B$ if and only if $A\<*B$.
\end{theorem}
\begin{proof}
In \cite[Theorem 2.3 and Theorem 2.4]{BaksalaryMitra} it is proved that $A\*<B$ if and only if $AA_{lr}=BA_{lr}$ and $A_{lr}A=A_{lr}B$ for some reflexive least square $g$-inverse of $A$. In the same theorems, it is proved that $A\<*B$ if and only if $AA_{mr}=BA_{mr}$ and $A_{mr}A=A_{mr}B$ for some reflexive minimum norm $g$-inverse of $A$. Now the proof is a direct consequence of Theorem \ref{th151MPrls}.
\end{proof}

\section{1MP-inverse in ring}

Let $a\in R$ be regular with $g$-inverse $a^-$. Let $p=aa^-$ and $q=a^-a$. Then $p$ and $q$ are idempotents.
Since $a=paq$, it follows that the matrix representations of $a$ with respect to $p$ and $q$ is
\begin{equation}\label{eq159}
a=\bmatrix a & 0 \\ 0 & 0 \endbmatrix_{p\times q},
\end{equation}
where $a$ is $(p,q)$-invertible with $a^-_{p,q}=a^-aa^-$. Thus, if $a^-$ is a reflexive g-inverse of $a$ then $a^-_{p,q}=a^-$.
Let $x=\bmatrix x_1 & x_2 \\ x_3 & x_4 \endbmatrix_{q\times p} \in a\{1\}$. We see that $axa=a$ if and only if $ax_1a=a$. Multiplying this equation by $a^-$ from the both sides, we get $qx_1p=a^-aa^-$, so $x_1=a^-aa^-$. On the other hand, if $x_1=a^-aa^-$ then $ax_1a=a$.
Therefore
\begin{equation}\label{eq15all inverses}
a\{1\}=\left\{\bmatrix a^-aa^- & x_2 \\ x_3 & x_4 \endbmatrix_{q\times p}: x_2\in qR(1-p), x_3\in (1-q)Rp, x_4\in (1-q)R(1-p)  \right\}.
\end{equation}
Similarly, we can prove that
\begin{equation}\label{eq15all reflexive inverses}
a\{1,2\}=\left\{\bmatrix a^-aa^- & x_2 \\ x_3 & x_3ax_2 \endbmatrix_{q\times p}: x_2\in qR(1-p), x_3\in (1-q)Rp  \right\}.
\end{equation}
If $a\in R^\dag$ and $p=aa^\dag$, $q=a^\dag a$ then $a^\dag =\bmatrix a^\dag & 0 \\ 0 & 0 \endbmatrix_{q\times p}$.
We can define a binary relation $\sim_l$ on $a\{1\}$ in the same way as it is done for matrices. For $a^-,a^=\in a\{1\}$,
$$a^-\sim_l a^= \;\;\Leftrightarrow\;\; a^-a=a^=a.$$
The relation $\sim_l$ is an equivalence relation on $a\{1\}$. If $a^-\in a\{1\}$ then by (\ref{eq15all inverses}),
$$a^-=\bmatrix a^\dag & x_2 \\ x_3 & x_4 \endbmatrix_{q\times p}$$
for some $x_2\in qR(1-p)$, $x_3\in (1-q)Rp$ and $x_4\in (1-q)R(1-p)$. We have
$$a^-a=\bmatrix q & 0 \\ x_3a & 0 \endbmatrix_{q\times q}.$$
Similarly, for $a^= =\bmatrix a^\dag & y_2 \\ y_3 & y_4 \endbmatrix_{q\times p}\in a\{1\}$ we have $a^=a=\bmatrix q & 0 \\ y_3a & 0 \endbmatrix_{q\times q}$ for some $y_3\in (1-q)Rp$. It follows that $a^-\sim_l a^=$ if and only if $x_3a=y_3a$. Multiplying this by $a^\dag$ from the right, we obtain $x_3p=y_3p$, ie. $x_3=y_3$. Therefore, $a^-\sim_l a^=$ if and only if $x_3=y_3$. It follows that we can choose the element
$$\bmatrix a^\dag & 0 \\ x_3 & 0 \endbmatrix_{q\times p}$$
as  the most natural representative of the class $[a^-]_{\sim_l}$. This representative is equal to $a^-aa^\dag$ because
\begin{equation}\label{eq154}
a^-aa^\dag=\bmatrix a^\dag & x_2 \\ x_3 & x_4 \endbmatrix_{q\times p}\bmatrix a & 0 \\ 0 & 0 \endbmatrix_{p\times q}\bmatrix a^\dag & 0 \\ 0 & 0 \endbmatrix_{q\times p}=\bmatrix a^\dag aa^\dag & 0 \\ x_3aa^\dag & 0 \endbmatrix_{q\times p}=\bmatrix a^\dag & 0 \\ x_3 & 0 \endbmatrix_{q\times p}
\end{equation}
We can now extend the definition of 1MP-inverse from the matrix case to the ring case.

\begin{definition}
Let $a\in R^\dag$ and choose $a^-\in a\{1\}$. The element
$$a^{-\dag}=a^-aa^\dag$$
is called a 1MP-inverse of $a$. The set of all 1MP-inverses of $a$ is denoted by $a\{-\dag\}$.
\end{definition}
Thus, we have
$$a\{-\dag\}=\{a^-aa^\dag : a^-\in a\{1\}\}.$$

Theorem \ref{th151} is also valid in an arbitrary ring with involution.

\begin{theorem}\label{th153}(See \cite[Theorem 3.1]{1MP} for the matrix case.)
For $a\in R^\dag$ and $g\in R$ the following statements are equivalent:
\begin{enumerate}[(i)]
\item $g\in a\{-\dag\}$;
\item $g$ is the solution of the system of equation
$$xax=x, \quad ax=aa^\dag;$$
\item $g\in a\{1,2,3\}$.
\end{enumerate}
\end{theorem}
\begin{proof}
(i) $\Rightarrow$ (ii): If $g\in a\{-\dag\}$ then $g=a^-aa^\dag$ for some $a^-\in a\{1\}$. We have
$$gag=a^-aa^\dag a a^-aa^\dag= a^-aa^- aa^\dag=a^-aa^\dag=g$$
and
$$ag=aa^-aa^\dag=aa^\dag.$$

(ii) $\Rightarrow$ (iii): If $gag=g$ and $ag=aa^\dag$ then $aga=a$ and $(ag)^*=(aa^\dag)^*=aa^\dag=ag$, so $g\in a\{1,2,3\}$.

(iii) $\Rightarrow$ (i): Suppose that $aga=a$, $gag=g$ and $(ag)^*=ag$. We obtain
$$gaa^\dag=gagaa^\dag=g(ag)^*(aa^\dag)^*=g(aa^\dag ag)^*=g(ag)^*=gag=g.$$
Therefore, $g=a^-aa^\dag$ for $a^-=g\in a\{1\}$.
\end{proof}

\begin{remark}\label{rem151}
When we work with general rings then we don't have a guarantee that the existence of a $\{1,2,3\}$ inverse of $a$ implies the existence of $a^\dag$. Because of that, we will primarily focus on the set $a\{1,2,3\}$ rather then $a\{-\dag\}$, although, by previous theorem, these sets coincide when $a$ is Moore-Penrose invertible. Note that if $g\in a\{1,3\}$ then $gag\in a\{1,2,3\}$. Therefore
$$R^{(1,2,3)}=R^{(1,3)}.$$
\end{remark}

The characterization similar to characterization (\ref{eq153}) in the next theorem was given for complex matrices in \cite[Proposition 3.2]{1MP}. The characterization (\ref{eq155}) is a generalization of Theorem \ref{th151MPrls}.

\begin{theorem}\label{th154}
Let $a\in R^{(1,3)}$ and let $p=ah$ and $q=ha$ where $h$ is fixed $\{1,2,3\}$-inverse of $a$. Then we have
\begin{align}
a\{1,2,3\}&=a\{1\}aa\{1,3\}\; \label{eq151}  \\
a\{1,2,3\}&=\left\{\bmatrix h & 0 \\ u & 0 \endbmatrix_{q\times p}: u\in (1-q)Rp\right\}\label{eq152} \\
&=\{h+(1-ha)wah : w\in R\}\label{eq153} \\
a\{1,2,3\}&=\{(a^*a)^-a^*:(a^*a)^-\in (a^*a)\{1\}\}; \label{eq155}.
\end{align}
\end{theorem}
\begin{proof}
Let us prove (\ref{eq151}). If $g\in a\{1,2,3\}$ then $g=gag\in a\{1\}aa\{1,3\}$. Conversely, suppose that $g=a^-aa^{(1,3)}$ for some $a^-\in a\{1\}$ and $a^{(1,3)}\in a\{1,3\}$. We have
\begin{align*}
aga&=aa^-aa^{(1,3)}a=a \\
gag&=a^-aa^{(1,3)}aa^-aa^{(1,3)}=a^-aa^{(1,3)}=g \\
ag&=aa^-aa^{(1,3)}=aa^{(1,3)}=(aa^{(1,3)})^*=(ag)^*,
\end{align*}
so $g\in a\{1,2,3\}$.

Let us prove (\ref{eq152}). By (\ref{eq15all reflexive inverses}), $g\in a\{1,2\}$ if and only if $g=\bmatrix h & x_2 \\ x_3 & x_3ax_2 \endbmatrix_{q\times p}$. Since $ag=\bmatrix p & ax_2 \\ 0 & 0 \endbmatrix_{p\times p}$, we obtain that $ag=(ag)^*$ if and only if $ax_2=0$ which is equivalent with $x_2=0$ because $x_2=qx_2=hax_2$. It follows that $g\in a\{1,2,3\}$ if and only if $g=\bmatrix h & 0 \\ x_3 & 0 \endbmatrix_{q\times p}$.
Note that the characterization (\ref{eq152}) is just the matrix record of characterization (\ref{eq153}).

Finally, let us prove the equality (\ref{eq155}). Let $g\in a\{1,2,3\}$ and set $s=gg^*$. We have $s\in (a^*a)\{1\}$ since
$$a^*agg^*a^*a=a^*(ag)^*g^*a^*a=(agaga)^*a=a^*a.$$
Also, $sa^*=gg^*a^*=g(ag)^*=g$, so $g$ belongs to the right-hand side of (\ref{eq155}).

On the other hand, if $g=(a^*a)^-a^*$ for some $(a^*a)^-\in (a^*a)\{1\}$ then
\begin{align*}
aga&=a(a^*a)^-a^*a=aha(a^*a)^-a^*a=(ah)^*a(a^*a)^-a^*a\\
&=h^*a^*a(a^*a)^-a^*a=h^*a^*a=a \\
gag&=(a^*a)^-a^*a(a^*a)^-a^*=(a^*a)^-a^*a(a^*a)^-(aha)^*\\
&=(a^*a)^-a^*a(a^*a)^-a^*ah=(a^*a)^-a^*ah=(a^*a)^-a^*=g \\
ag&=a(a^*a)^-a^*=ah a(a^*a)^-(aha)^*=(ah)^*a(a^*a)^-a^*(ah)^* \\
&=h^*a^*a(a^*a)^-a^*ah=h^*a^*ah=(ah)^*ah=ah,
\end{align*}
so $g\in a\{1,2,3\}$.
\end{proof}

\section{$<^{-\dag}$ order and left-star order in a ring}

The relations $<^{-\dag}$ and $<^{\dag -}$ can be extended from the complex matrix case (equations (\ref{eq15<-+}) and (\ref{eq15<+-})) to a general ring case in the straightforward way. In accordance with Remark \ref{rem151} we give the following definition.

\begin{definition}
For $a,b\in R$ we say that $a$ is lower then or equal to $b$ with respect to $<^{-\dag}$, which is denoted by $a<^{-\dag}b$ if there exists $g\in a\{1,2,3\}$ such that
$$ag=bg \text{ and } ga=gb.$$
\end{definition}

Recall the definition of the minus partial order.

\begin{definition}(\cite{How_to_partialy_order})
For $a,b\in R$ we say that $a<^-b$ if there exists $a^-\in a\{1\}$ such that
$$aa^-=ba^-\; \text{ and } a^-a=a^-b.$$
\end{definition}
The relation $<^-$ is a partial order relation on $R^{(1)}$, \cite[Theorem 1]{How_to_partialy_order}.

Note that if $ag=bg$ and $ga=gb$ for some $g\in a\{1,3\}$ then $h:=gag\in a\{1,2,3\}$ and $ah=bh$, $ha=hb$. We obtain the following result.
\begin{proposition}
Let $a,b\in R$. Then
$$a<^{-\dag} b \; \Leftrightarrow \; ag=bg \text{ and } ga=gb \text{ for some } g\in a\{1,3\}.$$
Also,
$$a<^{-\dag}b \; \Rightarrow \; a<^-b.$$
\end{proposition}

It follows that $<^{-\dag}$ is reflexive and antisymmetric relation on $R^\dag$.

We can define left-star partial order in a ring by analogy with the matrix case, see definitions in (\ref{eq15*<}). Some additional explanations will be given in the Section 6.

\begin{definition}\label{def151}
For $a,b\in R$ we say that $a\*< b$ if
$$a^*a=a^*b \text{ and } aR\subseteq bR.$$

\end{definition}
Note that the condition $aR\subseteq bR$ is equivalent with $a=bc$ for some $c\in R$ since $R$ has the multiplicative identity. Suppose that $a\*< b$. Then $a^*a=a^*b=b^*a$ and $a=bc$. It is not difficult to show that $(b-a)^*(b-a)=(b-a)^*b$ and $b-a=b(1-c)\in bR$, so $b-a\*< b$. It follows that
\begin{equation}\label{eq1510}
a\*< b \; \Leftrightarrow \; (b-a)\*< b.
\end{equation}

We will prove in the next theorem that, as in the matrix case, the relation $<^{-\dag}$ coincides with relation $\*<$.

\begin{theorem}\label{th155}
Let $a\in R^{(1,3)}$ and $b\in R$. Then
$$a<^{-\dag}b \; \Leftrightarrow \; a\*< b.$$
\end{theorem}
\begin{proof}
Suppose that $a<^{-\dag}b$. There exists $g\in a\{1,2,3\}$ such that $ag=bg$ and $ga=gb$. It follows that $a=aga=bga$, so $aR\subseteq bR$. Also
$$a^*b=(aga)^*b=a^*agb=a^*aga=a^*a,$$
and thus $a\*<b$. Suppose now that $a\*<b$, that is, $a^*a=a^*b$ and $a=bc$ for some $c\in R$. Fix $h\in a\{1,2,3\}$ and set $p=ah$ and $q=ha$. As before, $p=p^*$ and
$$a=\bmatrix a & 0 \\ 0 & 0 \endbmatrix_{p\times q}, \quad h=\bmatrix h & 0 \\ 0 & 0 \endbmatrix_{q\times p}.$$
Suppose that
$$b=\bmatrix b_1 & b_2 \\ b_3 & b_4 \endbmatrix_{p\times q}, \quad c=\bmatrix c_1 & c_2 \\ c_3 & c_4 \endbmatrix_{q\times q}.$$
From
$$a^*a=\bmatrix a^*a & 0 \\ 0 & 0 \endbmatrix_{q^*\times q}=\bmatrix a^*b_1 & a^*b_2 \\ 0 & 0 \endbmatrix_{q^*\times q}=a^*b$$
we obtain $a^*a=a^*b_1$ and $a^*b_2=0$. Therefore, $b_1=pb_1=ah b_1=h^*a^*b_1=h^*a^*a=pa=a$ and similarly $b_2=0$. From $a=bc$ we obtain
$$\bmatrix a & 0 \\ 0 & 0 \endbmatrix_{p\times q}=\bmatrix a & 0 \\ b_3 & b_4 \endbmatrix_{p\times q}\bmatrix c_1 & c_2 \\ c_3 & c_4 \endbmatrix_{q\times q}=\bmatrix ac_1 & ac_2 \\ b_3c_1+b_4c_3  & b_3c_2+b_4c_4 \endbmatrix_{p\times q}.$$
As in the previous part of the proof, we can easily find that $c_1=q$, $c_2=0$ and $b_3+b_4c_3=0$. Note that $c_3h\in (1-q)Rp$. Set
$$g=\bmatrix h & 0 \\ c_3h & 0 \endbmatrix_{q\times p}.$$
By Theorem \ref{th154}, we have that $g\in a\{1,2,3\}$. The direct calculation shows that
$$bg=\bmatrix a & 0 \\ b_3 & b_4 \endbmatrix_{p\times q}\bmatrix h & 0 \\ c_3h & 0 \endbmatrix_{q\times p}=\bmatrix p & 0 \\ b_3h +b_4c_3h & 0 \endbmatrix_{p\times p}=\bmatrix p & 0 \\ 0 & 0 \endbmatrix_{p\times p}=ag.$$
Similarly
$$ga=gb=\bmatrix q & 0 \\ c_3 & 0 \endbmatrix_{q\times q}.$$
It follows by definition that $a<^{-\dag}b$.
\end{proof}

Let us note the following easy observation. Suppose that $a\in R^\dag$ and $b\in R$. If $a^*a=a^*b$ then $a^\dag b=a^\dag(a^\dag)^*a^*b=a^\dag(a^\dag)^*a^*a=a^\dag a$. Conversely, if $a^\dag a=a^\dag b$ then $a^*b=a^*aa^\dag b=a^*aa^\dag a=a^*a$. Therefore,
\begin{equation}
a^*a=a^*b \; \Leftrightarrow \; a^\dag a=a^\dag b.
\end{equation}
From Theorem \ref{th155} it follows that
$$a<^{-\dag}b \; \Leftrightarrow \; a^\dag a=a^\dag b \text{ and } aR\subseteq bR.$$

\begin{remark}
Since Theorem \ref{th155} shows the equality of relations $<^{-\dag}$ and $\*<$, we will exclusively use the mark $\*<$ for both relations. Of course that we are in the position to choose which of the two definitions we will use for $\*<$ as needed.
\end{remark}

The theory of matrix partial orders based on generalized inverses is well developed, see monograph \cite{knjiga}. The unified theory of these relations was introduced by Mitra in \cite{Unified_theory}. This unified theory has recently been generalized in an arbitrary ring context by Raki\'c in \cite{Jedinstvena_teorija}. By the end of this section we will consider our relations through the prism of this unified theory.

First we need to introduce some notions which is given in \cite{Unified_theory} and \cite{Jedinstvena_teorija}. If $\mathcal{P}(R)$ denotes the power set of $R$ then a function $\mathcal{G}:R\rightarrow \mathcal{P}(R)$ is called a $g$-map if for every $a\in R$, $\mathcal{G}(a)$ is a certain subset of $a\{1\}$. The set $\Omega_\mathcal{G}=\{a\in R : \mathcal{G}(a)\neq \emptyset\}$ is called the support of the $g$-map $\mathcal{G}$. 

For $a,b\in R$ and a $g$-map $\mathcal{G}$ we say
$$a<^\G b \quad \text{if} \quad a\in\Omega_\G,\, ga=gb \text{ and } ag=bg \quad \text{for some} \quad g\in \G(a).$$

We will focus on the $g$-map $\mathcal{G}(a)=a\{1,2,3\}$, for which $\Omega_\mathcal{G}=R^{(1,2,3)}=R^{(1,3)}$.

For $a\in R$ and $g$-map $\mathcal{G}$, the class
$$\tilde{\G}(a)=\{g\in R : ga=ha, \, ag=ah \text{ for some } h\in \G(a)\}$$
is called the completion of $\mathcal{G}(a)$. Let us denote by $\tilde{a}\{i_1,\dots,i_n\}$ the completion of $a\{i_1,\dots,i_n\}$.

If $\mathcal{G}$ is a $g$-map then we say that the pair $(a,b)$, $a,b\in R$ satisfies the (T)-condition if $hah\in \mathcal{G}(a)$ for every $h\in \mathcal{G}(b)$.

We say that a $g$-map $\mathcal{G}$ is semi-complete if for every $a\in R$, the pair $(a,a)$ satisfies the (T)-condition.

\begin{proposition}
The $g$-map $\mathcal{G}(a)=a\{1,2,3\}$ is semi-complete and its completion is $a\{1,3\}$.
\end{proposition}
\begin{proof}
Let $\mathcal{G}(a)=a\{1,2,3\}$. If $g\in\tilde{\mathcal{G}}(a)$ then $ag=ah$ and $ga=ha$ for some $h\in a\{1,2,3\}$. From $ag=ah$ we obtain $aga=aha=a$ and $(ag)^*=(ah)^*=ah=ag$, so $g\in a\{1,3\}$. Suppose now that $g\in a\{1,3\}$. Then $h:=gag\in a\{1,2,3\}$ and $ah=agag=ag$, $ha=gaga=ga$, so $g\in\tilde{\mathcal{G}}(a)$. Since $h=hah$ for every $h\in \mathcal{G}(a)$, we conclude that $\mathcal{G}$ is semi-complete.
\end{proof}

In the next theorem, for a fixed $a\in R^{(1,3)}$, we will characterize all elements which is greater then $a$ with respect to $\*<$ order.

\begin{theorem}\label{th1514}
For $a\in R^{(1,3)}$ we have
  \begin{equation*}
  \{b\in R : a\*< b\}=\{a+(1-ag)d(1-ga) : g\in a\{1,2,3\},\, d\in R\}.
  \end{equation*}
That is, $a\*< b$ if and only if there exists $g\in a\{1,2,3\}$ such that
$$b=\bmatrix a & 0 \\ 0 & v \endbmatrix_{p\times q},$$
for some $v\in (1-p)R(1-q)$, where $p=ag$ and $q=ga$.
\end{theorem}
\begin{proof}
Since the $g$-map $\mathcal{G}(a)=a\{1,2,3\}$ is semi-complete, the proof is a direct consequence of Theorem \ref{th155} and Theorem 3.2 and Corollary 3.3 in \cite{Jedinstvena_teorija}.
\end{proof}

The following result was originally proved for complex matrices in Theorem 2.1 in \cite{BaksalaryBaksalary}. The same result for $<^{-\dag}$ order has recently been proved in \cite{1MP}.

\begin{theorem}\label{th1515}
Let $a\in R^{(1,3)}$, $h\in a\{1,2,3\}$, and $p=ah$, $q=ha$. Then $a\*< b$ if and only if
$$b=\bmatrix a & 0 \\ b_4u & b_4 \endbmatrix_{p\times q},$$
for some $b_4\in (1-p)R(1-q)$ and some $u\in Rq$.
\end{theorem}
\begin{proof}
If $a\*< b$ then $ag=bg$ and $ga=gb$ for some $g=\bmatrix h & 0 \\ v & 0 \endbmatrix_{q\times p}\in a\{1,2,3\}$ (Theorem \ref{th154}). Of course, $a=\bmatrix a & 0 \\ 0 & 0 \endbmatrix_{p\times q}$. Let $b=\bmatrix b_1 & b_2 \\ b_3 & b_4 \endbmatrix_{p\times q}$. From $ga=gb$ we obtain $hb_1=q$ and $hb_2=0$, so $b_1=a$ and $b_2=0$. From $ag=bg$ we obtain $b_3h+b_4v=0$. Multiplying this by $a$ from the right, we get $b_3=-b_4va=b_4u$ for $u=-va\in Rq$. Conversely, let $b=\bmatrix a & 0 \\ b_4u & b_4 \endbmatrix_{p\times q}$ where $b_4\in (1-p)R(1-q)$ and $u\in Rq$ are arbitrary. For
$$g=\bmatrix h & 0 \\ -(1-q)uh & 0 \endbmatrix_{q\times p}$$
we know that $g\in a\{1,2,3\}$ and
it is easy to show that $ag=bg$ and $ga=gb$, so $a\*< b$.
\end{proof}

It follows by (\ref{eq159}) that both in Theorem \ref{th1514} and in Theorem \ref{th1515} the element $a$ in matrix representations is $(p,q)$-invertible.

The following theorem gives the most natural matrix representations of $a$ and $b$ when $a\*< b$, because all elements appearing in two matrices are invertible in a sense of Definition \ref{def15p,q_invertibilnost}.

\begin{theorem}\label{th158}
  Let $a,b\in R^{(1,3)}$. Fix $h\in b\{1,2,3\}$ and set
    \begin{align*}
    &p_1=ah, \quad p_2=(b-a)h, \quad p_3=1-bh\\
    &q_1=ha, \quad q_2=h(b-a), \quad q_3=1-hb.
    \end{align*}
The following statements are equivalent
\begin{enumerate}[{\rm (i)}]
\item $a\*< b$;

\item The equalities
$$1=p_1+p_2+p_3 \quad \text{and} \quad 1=q_1+q_2+q_3$$
are respectively an orthogonal decomposition and a decomposition of the identity of the ring $R$ with respect to which $a$ and $b$ have the following matrix forms:
    \begin{equation}\label{decomposition main}
      a=\left[
      \begin{array}{ccc}
        a & 0 & 0 \\
        0 & 0 & 0 \\
        0 & 0 & 0
      \end{array}
      \right]_{p\times q}, \quad b=\left[
      \begin{array}{ccc}
        a & 0 & 0 \\
        0 & b-a & 0 \\
        0 & 0 & 0
      \end{array}
      \right]_{p\times q},
    \end{equation}
\end{enumerate}
where $a$ is $(p_1,q_1)$-invertible with $a^-_{p_1,q_1}=hah$ and $b-a$ is $(p_2,q_2)$-invertible with $(b-a)^-_{p_2,q_2}=h-hah$.
\end{theorem}
\begin{proof}
The most of the theorem directly follows by Theorem 3.1 in \cite{Jedinstvena_teorija} where the similar result is proven for arbitrary semi-complete $g$-map. We will only prove here the results which are not included in this theorem (which is specific for $\*<$ order). Thus we have to prove that decomposition of the identity $1=p_1+p_2+p_3$ is orthogonal and we have to prove the pseudo invertibility of $a$ and $b-a$.
Suppose that $a\*< b$. Then $ag=bg$ and $ga=gb$ for some $g\in a\{1,2,3\}$. The decomposition of the identity $1=p_1+p_2+p_3$ is orthogonal because the idempotent $p_1$ is self-adjoint:
\begin{equation}\label{eq1512}
p_1=ah=agah=agbh=(bhag)^*=(bhbg)^*=(bg)^*=(ag)^*=ag.
\end{equation}
By (\ref{eq1512}), the direct check shows that $a^-_{p_1,q_1}=q_1hp_1=hah$ and $(b-a)^-_{p_2,q_2}=q_2hp_2=h-hah$:
\begin{align*}
ahah&=agah=ah=p_1 \\
haha&=haga=ha=q_1 \\
(b-a)(h-hah)&=bh-bhah-ah+ahah=bh-bhag-ah+agah \\
&=bh-(agbh)^*=bh-(ah)^*=bh-ah=p_2 \\
(h-hah)(b-a)&=hb-ha-hahb+haha=hb-hahb\\
&=hb-hagb=hb-haga=q_2.
\end{align*}
\end{proof}

The following characterization is typical for many partial orders based on $g$-inverses. The similar result for minus partial order was presented in \cite{Rakic} and for star, sharp, core and dual core partial order in \cite{Star_sharp_core_order_in_ring}.

\begin{theorem}\label{th1511}
Let $a\in R^{(1,3)}$ and $b\in R$. Then $a\*< b$ if and only if there exist self-adjoint idempotent $p$ and idempotent $q$ such that $a=pb=bq$.
\end{theorem}
\begin{proof}
If $a\*< b$ then $ag=bg$ and $ga=gb$ for some $g\in a\{1,2,3\}$. If $p=ag$ and $q=ga$ then $p$ is self-adjoint idempotent, $q$ is idempotent and $a=aga=agb=pb$ and $a=bga=qa$. If $a=pb=bq$ where $p$ is self-adjoint idempotent and $q$ is idempotent then $a=pa$, $aR\subseteq bR$ and $a^*b=(pa)^*b=a^*pb=a^*a$.
\end{proof}

In the study of partial orders based on generalized inverses, it is usual to examine its relationships with inclusions of appropriate subsets of $g$-inverses. This is already done for the minus, star, sharp, core and dual core partial orders in \cite{Rakic} and \cite{Star_sharp_core_order_in_ring}. The next two theorems show that the left-star partial order is not exception in this respect.

\begin{theorem}\label{th156}
Let $a,b\in R^{(1,3)}$ and suppose that $a\*< b$. Then
\begin{enumerate}[(i)]
\item The pair $(a,b)$ satisfies the (T)-condition, that is,
$$hah\in a\{1,2,3\}, \; \forall h\in b\{1,2,3\}.$$

\item $b\{1,3\}\subseteq a\{1,3\}$.

\end{enumerate}
\end{theorem}
\begin{proof}
Suppose that $a\*< b$. Since the $g$-map $\mathcal{G}(a)=a\{1,2,3\}$ is semi-complete and since $b\{1,3\}$ is the completion of $b\{1,2,3\}$, it follows by Corollary 3.6 in \cite{Jedinstvena_teorija} that the conditions (i) and (ii) are equivalent. Therefore, it is enough to prove only the inclusion (ii). There exists $g\in a\{1,2,3\}$ such that $ag=bg$ and $ga=gb$. Let $h\in b\{1,3\}$. Then we have
\begin{align*}
ah&=agah=agbh=(bhag)^*=(bhbg)^*=(bg)^*=(ag)^*=ag \\
aha&=aga=a,
\end{align*}
so $h\in a\{1,3\}$.
\end{proof}

The following observation follows by (\ref{eq1510}) and Theorem \ref{th156}. Suppose that $a,b\in R^{(1,3)}$ and $a\*< b$. Then for every $h\in b\{1,2,3\}$
$$h(b-a)h=h-hah=(b-a)^-_{p_2,q_2}\in (b-a)\{1,2,3\}.$$
In particular, we conclude that $b-a\in R^{(1,3)}$.

From the unified theory of $g$-based partial orders it follows that any of the conditions (i) or (ii) from previous Theorem \ref{th156} is sufficient to conclude that $\*<$ is a partial order relation on $R^{(1,3)}$, see Corollary 3.6 in \cite{Jedinstvena_teorija}.
For the reader's convenience we will postpone this conclusion because it directly follows by the following important result.

Theorem \ref{th157} is originally proved by Baksalary and Mitra in \cite{BaksalaryMitra}. Another proof of the same result can be found in Theorem 6.5.17 in \cite{knjiga}. Our proof is in the spirit of that proof. The starting point of the proof in \cite{knjiga} is the singular value decomposition of the matrix $B$. We cannot use the analogy in our case, because we do not suppose the existence of the Moore-Penrose inverse of $b$. This is one among the reasons that we cannot "imitate" that proof.

\begin{theorem}\label{th157}
Let $a,b\in R^{(1,3)}$. Then
$$a\*< b \; \Leftrightarrow \; b\{1,3\}\subseteq a\{1,3\}.$$
\end{theorem}
\begin{proof}
We have proved the "if" part in Theorem \ref{th156}. Suppose that $b\{1,3\}\subseteq a\{1,3\}$. Fix an $h\in b\{1,2,3\}$ and let $p=bh=(bh)^*$ and $q=hb$. Suppose that $a=\bmatrix a_1 & a_2 \\ a_3 & a_4 \endbmatrix_{p\times q}$. Like in Theorem \ref{th154}, we can show that $g\in b\{1,3\}$ if and only if
$$g=\bmatrix h & 0 \\ x_3 & x_4 \endbmatrix_{q\times p}$$
for some $x_3\in (1-q)Rp$ and $x_4\in (1-q)R(1-p)$. Since $b\{1,3\}\subseteq a\{1,3\}$, we have that $aga=a$ and $(ag)^*=ag$ for every $x_3$ and $x_4$. If we take $x_3=0$ and $x_4=0$ then the condition $(ag)^*=ag$ gives $(a_1h)^*=a_1h$ and $a_3h=0$. Thus, $a_3=a_3q=0$. The condition $aga=a$ gives
$$\bmatrix a_1ha_1 & a_1ha_2 \\ 0 & 0 \endbmatrix_{p\times q}=\bmatrix a_1 & a_2 \\ 0 & a_4 \endbmatrix_{p\times q}.$$
Therefore,
\begin{equation}\label{eq156}
a_4=0, \; a_1ha_1=a_1 \; \text{ and } \; a_1ha_2=a_2.
\end{equation}
For arbitrary $g\in b\{1,3\}$ the conditions $ag=(ag)^*$ and $aga=a$ provide
$$\bmatrix a_1h+a_2x_3 & a_2x_4 \\ 0 & 0 \endbmatrix_{p\times p}=\bmatrix a_1h+(a_2x_3)^* & 0 \\ (a_2x_4)^* & 0 \endbmatrix_{p\times p}$$
and
$$\bmatrix a_1ha_1+a_2x_3a_1 & a_1ha_2+a_2x_3a_2 \\ 0 & 0 \endbmatrix_{p\times q}=\bmatrix a_1 & a_2 \\ 0 & 0 \endbmatrix_{p\times q}.$$
Therefore,
$$a_2x_4=0,\; (a_2x_3)^*=a_2x_3, \; a_2x_3a_1=0, \; a_2x_3a_2=0$$
for every $x_3\in (1-q)Rp$ and for every $x_4\in (1-q)R(1-p)$.
It follows that
$$a_2x_3=a_1ha_2x_3=(a_2x_3a_1h)^*=0,$$
Take $x_3=(1-q)p$ and $x_4=(1-q)(1-p)$ to conclude that
$$a_2=a_2(1-q)=a_2(1-q)p+a_2(1-q)(1-p)=a_2x_3+a_2x_4=0.$$
We have proved that
$$a=\bmatrix a_1 & 0 \\ 0 & 0 \endbmatrix_{p\times q},$$
so $a=a_1\in pRq$ and $aha=a$, $(ah)^*=ah$. Hence, $a=pa=bha$, so $aR\subseteq bR$. Also,
$$a^*b=(aha)^*b=a^*ahb=a^*aq=a^*a.$$
It follows that $a\*< b$ and the proof is complete.
\end{proof}

\begin{corollary}
The relation $\*<$, that is, the relation $<^{-\dag}$, is a partial order relation on $R^{(1,3)}$. 
\end{corollary}
\begin{proof}
We have already establish that the relations $\*<$ is reflexive and antisymmetric. The transitivity of $\*<$ follows by Theorem \ref{th157}, since the subset relation is transitive.
\end{proof}

\section{MP1-inverse, $<^{\dag -}$ order and the right-star order}

In this section we will present dual results concerning the MP1-inverse, the relation $<^{-\dag}$ and the right-star order. All results can be proved in a very similar (analogous) way as its duals. Because of that, we will omit the proofs.

Similarly as in the case of 1MP-inverse, the set of MP1-inverses can be introduced as the set of all canonical representatives of the the quotient space $a\{1\}/\sim_r$, where the relation $\sim_r$ is defined on $a\{1\}$ by $a^-\sim_r a^=$ if $aa^-=aa^=$, $a^-,a^=\in a\{1\}$. It turns out that we arrive to the following definition.

\begin{definition}
Let $a\in R^\dag$ and choose $a^-\in a\{1\}$. The element
$$a^{\dag -}=a^\dag aa^-$$
is called a MP1-inverse of $a$. The set of all MP1-inverses of $a$ is denoted by $a\{\dag -\}$.
\end{definition}

Note that $R^{(1,2,4)}=R^{(1,4)}$.

\begin{theorem}
Let $a\in R^{(1,4)}$ and let $p=ah$, $q=ha$ where $h$ is fixed $\{1,2,4\}$-inverse of $a$. Then we have
\begin{align*}
a\{1,2,4\}&= a\{1,4\}aa\{1\}  \\
&=\left\{\bmatrix h & u \\ 0 & 0 \endbmatrix_{q\times p}: u\in qR(1-p)\right\} \\
&=\{h+haw(1-h) : w\in R\} \\
&=\{a^*(aa^*)^-:(aa^*)^-\in (aa^*)\{1\}\}.
\end{align*}
If $a\in R^\dag$ then
$$a\{\dag -\}=a\{1,2,4\}=\{x: xax=x, \; xa=a^\dag a\}.$$
\end{theorem}

The following observation is evident
\begin{equation}\label{eq157}
g\in a\{1,2,3\} \; \Leftrightarrow \; g^*\in a^*\{1,2,4\}.
\end{equation}

The definitions of the $<^{-\dag}$ relation and the right-star order are expected.
\begin{definition}
For $a,b\in R$ we say that $a$ is lower then or equal to $b$ with respect to $<^{\dag -}$, which is denoted by $a<^{\dag -}b$ if there exists $g\in a\{1,2,4\}$ such that
$$ag=bg \text{ and } ga=gb.$$
\end{definition}

\begin{proposition}
Let $a,b\in R$. Then
$$a<^{\dag -} b \; \Leftrightarrow \; ag=bg \text{ and } ga=gb \text{ for some } g\in a\{1,4\}.$$
Also,
$$a<^{\dag -}b \; \Rightarrow \; a<^-b.$$
\end{proposition}

\begin{definition}
We say that $a\<* b$ if
$$aa^*=ba^*\; \text{ and } Ra\subseteq Rb.$$
\end{definition}

The relations $<^{\dag -}$ and $\<*$ coincides on $R^{(1,4)}$.

\begin{theorem}\label{th1510}
Let $a\in R^{(1,4)}$ and $b\in R$. Then
$$a<^{\dag -}b \; \Leftrightarrow \; a\<* b.$$
\end{theorem}

It follows from Theorem \ref{th1510} and (\ref{eq157}) that
$$a\*< b \; \Leftrightarrow a^*\<* b^*.$$

From
$$aa^*=ba^* \; \Leftrightarrow \; aa^\dag=ba^\dag$$
and Theorem \ref{th1510} it follows that
\begin{align*}
&a<^{\dag -}b \; \Leftrightarrow \; aa^\dag=ba^\dag \text{ and } Ra\subseteq Rb.
\end{align*}

Like for the $\*<$ order, the semi-completeness given in the next proposition allow us to transfer some results of the unified theory to the $\<*$ order relation.

\begin{proposition}
The $g$-map $\mathcal{G}(a)=a\{1,2,4\}$ is semi-complete and its completion is $a\{1,4\}$.
\end{proposition}

The next two theorems characterize all elements which are greater then $a\in R^{(1,4)}$ with respect to $\<*$ order.

\begin{theorem}
For $a\in R^{(1,4)}$ we have
  \begin{equation*}
  \{b\in R : a\<* b\}=\{a+(1-ag)d(1-ga) : g\in a\{1,2,4\},\, d\in R\}.
  \end{equation*}
That is, $a\<* b$ if and only if there exists $g\in a\{1,2,4\}$ such that
$$b=\bmatrix a & 0 \\ 0 & v \endbmatrix_{p\times q},$$
for some $v\in (1-p)R(1-q)$, where $p=ag$ and $q=ga$.
\end{theorem}

\begin{theorem}
Let $a\in R^{(1,4)}$, $h\in a\{1,2,4\}$, and $p=ah$, $q=ha$. Then $a\*< b$ if and only if
$$b=\bmatrix a & ub_4 \\ 0 & b_4 \endbmatrix_{p\times q},$$
for some $b_4\in (1-p)R(1-q)$ and $u\in pR$.
\end{theorem}

The following canonical matrix representations is characteristic for every $g$-based semi-complete relation.

\begin{theorem}\label{th1513}
  Let $a,b\in R^{(1,4)}$. Fix $h\in b\{1,2,4\}$ and set
    \begin{align*}
    &p_1=ah, \quad p_2=(b-a)h, \quad p_3=1-bh\\
    &q_1=ha, \quad q_2=h(b-a), \quad q_3=1-hb.
    \end{align*}
The following statements are equivalent
\begin{enumerate}[{\rm (i)}]
\item $a\<* b$;

\item The equalities
$$1=p_1+p_2+p_3 \quad \text{and} \quad 1=q_1+q_2+q_3$$
are respectively a decomposition and an orthogonal decomposition of the identity of the ring $R$ with respect to which $a$ and $b$ have the following matrix forms:
    \begin{equation}\label{decomposition main}
      a=\left[
      \begin{array}{ccc}
        a & 0 & 0 \\
        0 & 0 & 0 \\
        0 & 0 & 0
      \end{array}
      \right]_{p\times q}, \quad b=\left[
      \begin{array}{ccc}
        a & 0 & 0 \\
        0 & b-a & 0 \\
        0 & 0 & 0
      \end{array}
      \right]_{p\times q},
    \end{equation}
\end{enumerate}
where $a$ is $(p_1,q_1)$-invertible with $a^-_{p_1,q_1}=hah$ and $b-a$ is $(p_2,q_2)$-invertible with $(b-a)^-_{p_2,q_2}=h-hah$.
\end{theorem}

\begin{theorem}\label{th1512}
Let $a\in R^{(1,4)}$ and $b\in R$. Then $a\<* b$ if and only if there exist idempotent $p$ and self-adjoint idempotent $q$ such that $a=pb=bq$.
\end{theorem}

The relationships with inclusions of appropriate subsets of $g$-inverses is presented in the next theorem.

\begin{theorem}
Let $a,b\in R^{(1,4)}$. Then
$$a\<* b \; \Leftrightarrow \; b\{1,4\}\subseteq a\{1,4\}.$$
Also, if $a\<* b$ then the pair $(a,b)$ satisfies the (T)-condition, that is,
$$hah\in a\{1,2,4\}, \; \forall h\in b\{1,2,4\}.$$
\end{theorem}

As a direct consequence of previous theorem, we conclude that $\<*$ is transitive. We already know that it is reflexive and antisymmetric, so it is a partial order relation.

\begin{corollary}
The relation $\<*$, that is, the relation $<^{\dag -}$, is a partial order relation on $R^{(1,4)}$. 
\end{corollary}

It is natural to see the connection between one-sided star orders and star order $<^*$ which is defined by Drazin in \cite{Natural structures}:
$$a<^*b \; \Leftrightarrow \; aa^*=ba^* \text{ and } a^*a=a^*b.$$
Recall that $<^*$ is a partial order relation in arbitrary semigroup with proper involution, \cite{Natural structures}.

The star order has recently been examined in an arbitrary ring with involution in \cite{Star_sharp_core_order_in_ring}. It was shown that the relation $<^*$ is a partial order relation on $R^\dag$. The following is one of the basic characterization of the star order when $a,b\in R^\dag$, Theorem 2.6 in \cite{Star_sharp_core_order_in_ring}
\begin{equation}\label{eq158}
a<^*b \; \Leftrightarrow \; a=pb=br \text{ for some self-adjoint idempotents } p \text{ and } r.
\end{equation}

From theorems \ref{th1511} and \ref{th1512} and characterization (\ref{eq158}) we obtain the following expected result.

\begin{theorem}
Let $a\in R^\dag$ and $b\in R$. Then
$$a<^*b \; \Leftrightarrow \; a\*< b \text{ and } a\<* b.$$
\end{theorem}

\section{Concluding remarks}

We have considered several things in this paper. After examining the 1MP and MP1 inverses and associated partial orders $<^{-\dag}$ and $<^{\dag -}$ in the complex matrix case, we introduced and studied these notions in the context of an arbitrary ring with involution. Beside, we introduced the left-star $\*<$ and right-star $\<*$ partial orders by analogy with matrix case and showed that these orders coincide on $R^{(1,3)}$ with $<^{-\dag}$ and $<^{\dag -}$, respectively. After that we investigated the left-star order through the prism of unified theory of partial orders based on generalized inverses. The dual results of $\{1,2,4\}$-inverses, $<^{\dag -}$ order and right-star order are also presented.

It is worth to say that the orders $\*<$ and $\<*$ are also defined and studied in the context of Rickart $*$-rings, see Section 3 in \cite{Rikart zvezda} and \cite{Kremere}. The two generalizations are in the complete agreement because one can show that the left-star order in the sense of Definition \ref{def151} coincide with the left-star order given in Definition 10 in \cite{Rikart zvezda} when $a\in R^{(1,3)}$. This fact can be proved using characterization given in Theorem \ref{th1511}. But each of these two cases has its own specificity and its own proving techniques. Let us also mention the closely related diamond partial order which is defined for matrices in \cite{BaksalaryHauke} and also studied in rings in \cite{Thome3}.

The inspiration for most of the results and proofs in this paper come from the operator theory. For instance, the results that we develop can find nice applications in $B(H)$ - the algebra of bounded operators on Hilbert space $H$, in $C^*$-algebras, or in Rickart $*$-rings.
We will conclude this paper by indicating how the presented results can be applied in the case of Hilbert space operators.

Let $H$ and $K$ be Hilbert spaces and let $B(H,K)$ denote the set of all bounded linear operators from $H$ to $K$. Let $B(H)=B(H,H)$. Like in the matrix case, for $A\in B(H,K)$, we use $A^*$, $\R A$ and $\N A$ to denote respectively the Hilbert adjoint of $A$, the range and the null-space of $A$. As usual, $I$ stands for the identity operator. We write $H_1\dotplus H_2$ for the direct sum of subspaces and we write $H_1\oplus H_2$ for the orthogonal direct sum.
The different generalized inverses of $A$ are defined by the same equations as in the matrix (ring) case.
Recall that $A\in B(H,K)$ has a $g$-inverse if and only if it has the Moore-Penrose inverse if and only if $\R A$ is closed in $K$.

Let $A,B\in B(H,K)$ be operators with closed ranges. We can define relations $<^{-\dag}$ and $\*<$ like in the matrix case, using definitions (\ref{eq15<-+}) and (\ref{eq15*<}), respectively.
By Lemma 2.1 in \cite{prvi_rad} we know that $\R A\subseteq \R B$ if and only if $A=BC$ for some $C\in B(H)$. Because of that, from Theorem \ref{th155}, it follows that the relations defined by (\ref{eq15<-+}) and by (\ref{eq15*<}) are the same in the operator case as well. From the proof of Theorem \ref{th155} one can see that the technical obstacle that $B(H,K)$ is not a ring has no effect on the validity of this claim.

We will now give the interpretation of Theorem \ref{th158}. The interpretation of other results can be achieved by similar reasoning.
Let us follow the notation as in Theorem \ref{th158} and its proof.
By (\ref{eq1512}), we have
\begin{equation}\label{eq1511}
\R P_1=\R(AH)=\R(AG)=\R A.
\end{equation}
Recall that $A\*< B$ implies $B-A\*< B$ by (\ref{eq1510}). Now, by the same argument as in (\ref{eq1511}), we conclude that
$\R P_2=\R(B-A)$.
From the same reason as in the proof of Theorem \ref{th1516}, we have that $\N H=\N B^*$. Therefore,
\begin{equation*}
\R P_3=\R(I-BH)=\N(BH)=\N H=\N B^*
\end{equation*}
In \cite{Decomposition_of_identity}, the connection of the decomposition of the identity $I$ of ring $B(K)$ and topological direct sum was presented in detail. From that connection, it follows that $K=\R P_1\oplus \R P_2\oplus \N P_3$, so
$$K=\R A\oplus \R(B-A)\oplus \N B^*,$$
where the above decomposition is an orthogonal direct sum decomposition of $K$.
Finally, note that $\R Q_3=\R(I-HB)=\N(HB)=\N B$. We are now in a position to restate Theorem \ref{th158} in the operator case.

\begin{theorem}
Let $A,B\in B(H,K)$ be operators with closed ranges. Then $A\*<B$ if and only if the following hold
\begin{enumerate}[(i)]
\item There exist closed subspaces $H_1$ and $H_2$ of $H$ such that
$H=H_1\dotplus H_2\dotplus \N B;$

\item $K=\R A\oplus \R(B-A)\oplus \N B^*$;

\item $A,B:H_1\dotplus H_2\dotplus \N B \rightarrow \R A\oplus \R(B-A)\oplus \N B^*$ have matrix representations
$$      A=\left[
      \begin{array}{ccc}
        A_1 & 0 & 0 \\
        0 & 0 & 0 \\
        0 & 0 & 0
      \end{array}
      \right], \quad B=\left[
      \begin{array}{ccc}
        A_1 & 0 & 0 \\
        0 & B_1 & 0 \\
        0 & 0 & 0
      \end{array}
      \right],$$
where $A_1\in B(H_1,\R A)$ and $B_1\in B(H_2,\R(B-A))$ are invertible operators.
\end{enumerate}
\end{theorem}

\textbf{Acknowledgment}

\vspace{2cm}

Dragan Raki\'c, Faculty of Mechanical Engineering, University of Ni\v s, Aleksandra Medvedeva 14, 18000 Ni\v s, Serbia

Email address: rakic.dragan@gmail.com

\vspace{0.5cm}

Martin Ljubenovi\'c, Faculty of Mechanical Engineering, University of Ni\v s, Aleksandra Medvedeva 14, 18000 Ni\v s, Serbia

Email address: martinljubenovic@gmail.com

\end{document}